\documentclass[psamsfonts]{amsart}

\usepackage{amsmath,amssymb,amsfonts,amsthm}
\usepackage[all,arc]{xy}
\usepackage{enumerate}
\usepackage{mathrsfs}
\usepackage{aliascnt}
\usepackage{hyperref}

\usepackage{wasysym}
\usepackage{calligra}
\DeclareMathAlphabet{\lscr}{T1}{calligra}{m}{n}
\DeclareFontShape{T1}{calligra}{m}{n}{<->s*[1.1]callig15}{}

\DeclareFontFamily{OMS}{rsfs}{\skewchar\font'60}
\DeclareFontShape{OMS}{rsfs}{m}{n}{<-5>rsfs5 <5-7>rsfs7 <7->rsfs10 }{}
\DeclareSymbolFont{rsfs}{OMS}{rsfs}{m}{n}
\DeclareSymbolFontAlphabet{\scr}{rsfs}

\newcommand{\sA}{\scr{A}}
\newcommand{\sB}{\scr{B}}
\newcommand{\sC}{\scr{C}}

\newcommand{\sH}{\scr{H}}
\newcommand{\sI}{\scr{I}}

\newcommand{\sM}{\scr{M}}
\newcommand{\sN}{\scr{N}}
\newcommand{\sO}{\scr{O}}
\newcommand{\sP}{\scr{P}}

\newcommand{\sU}{\scr{U}}

\newcommand{\bC}{\mathbb{C}}

\newcommand{\bF}{\mathbb{F}}

\newcommand{\bR}{\mathbb{R}}

\newcommand{\bT}{\mathbb{T}}

\newcommand{\bZ}{\mathbb{Z}}

\def\quickop#1{\expandafter\newcommand\csname #1\endcsname{\operatorname{#1}}}
\quickop{Hom} \quickop{End} \quickop{Aut} \quickop{Tel} \quickop{Mic} 
\quickop{Ext} \quickop{Tor} \quickop{Id} \quickop{Coker} \quickop{Ker}
\quickop{Lim} \quickop{Colim} \quickop{Holim} \quickop{Hocolim}
\quickop{id} \quickop{tel} \quickop{mic} \quickop{coker} 
\quickop{colim} \quickop{holim} \quickop{hocolim} \quickop{im}

\DeclareMathAlphabet{\mathpzc}{OT1}{pzc}{m}{it}
\DeclareMathAlphabet{\mathfrc}{T1}{frc}{m}{n}
\DeclareMathAlphabet{\mathitbf}{OML}{cmm}{b}{it}
\usepackage{bbm} 

\quickop{trace}

\newcommand{\boxp}{\mathbin{\square}}
\newcommand{\oboxp}{\mathbin{\overline{\square}}}

\newcommand{\ground}{\mathbbm{k}}

\newcommand{\rmod}{\ground\textrm{-mod}}

\newcommand{\spaces}{\textrm{Top}}
\newcommand{\gtop}{\G\textrm{-Top}}

\newcommand{\op}{^{\text{op}}}
\newcommand{\fund}[1]{\Pi_\G{#1}}
\newcommand{\fundx}[1]{\Pi}
\newcommand{\fundhom}[2]{\fundx{}\!\left( {#1},{#2} \right)}
\newcommand{\chr}{\text{Ch}_\ground}
\newcommand{\fmap}[1]{(\alpha_{#1},[\gamma_{#1}])}

\newcommand{\fq}{\bF_q}

\newcommand{\cyctwo}{\bZ/2}
\newcommand{\cpv}[1]{\bC P({#1})}
\newcommand{\ab}{\textbf{Ab}}

\newcommand{\myspan}[3]{\xymatrix@R=0.5pc@C=1pc{ & {\gset{#2}} \ar[dr] \ar[dl] & \\ {\gset{#1}} & & {\gset{#3}}\\}}
\newcommand{\inlinespan}[3]{{\gset{#1}}\longleftarrow{\gset{#2}}\longrightarrow{\gset{#3}}}
\newcommand{\contraspan}[3]{\xymatrix@R=0.5pc@C=1pc{ & {#1} \ar@{=}[dr] \ar[dl]_{#3} & \\ {#2} & & {#1}\\}}
\newcommand{\cospan}[3]{\xymatrix@R=0.5pc@C=1pc{ & {#1} \ar@{=}[dl] \ar[dr]^{#3} & \\ {#1} & & {#2}\\}}
\newcommand{\idspan}[1]{\xymatrix@R=0.5pc@C=1pc{ & {#1} \ar@{=}[dl] \ar@{=}[dr] & \\ {#1} & & {#1}\\}}
\newcommand{\minicontraspan}[3]{\xymatrix@R=0.2pc@C=.4pc{ & {\scriptstyle #1} \ar@{=}[dr] \ar[dl]_{\scriptstyle #3} & \\ {\scriptstyle #2} & & {\scriptstyle #1}\\}}
\newcommand{\minicospan}[3]{\xymatrix@R=0.2pc@C=.4pc{ & {\scriptstyle #1} \ar@{=}[dl] \ar[dr]^{\scriptstyle #3} & \\ {\scriptstyle #1} & & {\scriptstyle #2}\\}}
\newcommand{\miniidspan}[1]{\xymatrix@R=0.2pc@C=.4pc{ & {\scriptstyle #1} \ar@{=}[dl] \ar@{=}[dr] & \\ {\scriptstyle #1} & & {\scriptstyle #1}\\}}

\newcommand{\mf}[5]{\xymatrix{ {#1} \ar@/_2ex/[d]_{#4} \\ {#2} \ar@(d,dr)[]_{#5} \ar@/_2ex/[u]_{#3} \\ }}
\newcommand{\mymatrix}[1]{\left( \begin{smallmatrix}#1 \end{smallmatrix} \right)}

\newcommand{\pullbackcorner}[1][dr]{\save*!/#1-1.2pc/#1:(-1,1)@^{|-}\restore}

\newcommand{\trivo}{\bullet}

\newcommand{\G}{G}
\newcommand{\eqchains}[1]{C_*^\G({#1})}
\newcommand{\sh}{h_\G} 
\newcommand{\ulh}{\underline{H}_\G} 
\newcommand{\redh}{\widetilde{H}_\G} 
\newcommand{\h}{{H}_\G} 

\newcommand{\coh}{\mathcal{H}}

\newcommand{\cohomplot}[9]{%
  \xymatrix@R=0.5pc@C=#9{
  &  &  &  &  &  & \makebox[0pt][c]{${#2}$} &  &  &  &  &  \\
{\phantom{\makebox[0pt][c]{${#8}$}}} &   & \makebox[0pt][c]{${#7}$} &   & \makebox[0pt][c]{${#7}$} &   & \makebox[0pt][c]{${#6}$} &   & . &   & . &   \\
   & . &   & . &   & . &   & . &   & . &   & . \\
{\phantom{\makebox[0pt][c]{${#8}$}}} &   & \makebox[0pt][c]{${#7}$} &   & \makebox[0pt][c]{${#7}$} &   & \makebox[0pt][c]{${#6}$} &   & . &   & . &   \\
   & . &   & . &   & . &   & . &   & . &   & . \\
    \ar@{<.>}[rrrrrrrrrrrr] &    & \makebox[0pt][c]{$#4$} &   & \makebox[0pt][c]{$#4$} &   & \makebox[0pt][c]{${#3}$} &   & \makebox[0pt][c]{$#5$} &   & \makebox[0pt][c]{$#5$} &   &  & \makebox[0pt][c]{${#1}$} \\
{\phantom{\makebox[0pt][c]{${#8}$}}} & . &   & . &   & . &   & . &   & \makebox[0pt][c]{${#7}$} &   & \makebox[0pt][c]{${#7}$} \\
{\phantom{\makebox[0pt][c]{${#8}$}}} &   & . &   & . &   & \makebox[0pt][c]{${#6}$} &   & . &   & . &   \\
{\phantom{\makebox[0pt][c]{${#8}$}}} & . &   & . &   & . &   & . &   & \makebox[0pt][c]{${#7}$} &   & \makebox[0pt][c]{${#7}$} \\
{\phantom{\makebox[0pt][c]{${#8}$}}} &   & . &   & . &   & \makebox[0pt][c]{${#6}$} &   & . &   & . &   \\
    &   &   &   &   &   & \ar@{<.>}[uuuuuuuuuu]  &   &   &   &   & \\
  } %
}

\newcommand{\shiftedpic}[8]{%
  \xymatrix@R=0.5pc@C=0.7pc{
  &  &  &  &  &  & \makebox[0pt][c]{${#2}$} &  &  &  &  &  \\
{\phantom{\makebox[0pt][c]{${#8}$}}} &   & \makebox[0pt][c]{${#7}$} &   & \makebox[0pt][c]{${#7}$} &   & \makebox[0pt][c]{${#6}$} &   & . &   & . &   \\
   & . &   & . &   & . &   & . &   & . &   & . \\
 &    & \makebox[0pt][c]{$#4$} &   & \makebox[0pt][c]{$#4$} &   & \makebox[0pt][c]{${#3}$} &   & \makebox[0pt][c]{$#5$} &   & \makebox[0pt][c]{$#5$} &   &  & \\
{\phantom{\makebox[0pt][c]{${#8}$}}} & . &   & . &   & . &   & . &   & \makebox[0pt][c]{${#7}$} &   & \makebox[0pt][c]{${#7}$} \\
{\phantom{\makebox[0pt][c]{${#8}$}}} &   & . &   & . &   & \makebox[0pt][c]{${#6}$} &   & . &   & . &   \\
{\phantom{\makebox[0pt][c]{${#8}$}}} & . &   & . &   & . &   & . &   & \makebox[0pt][c]{${#7}$} &   & \makebox[0pt][c]{${#7}$} \\
   \ar@{<.>}[rrrrrrrrrrrr] {\phantom{\makebox[0pt][c]{${#8}$}}} &   & . &   & . &   & \makebox[0pt][c]{${#6}$} &   & . &   & . &  & & \makebox[0pt][c]{${#1}$} \\
{\phantom{\makebox[0pt][c]{${#8}$}}} & . &   & . &   & . &   & . &   & \makebox[0pt][c]{${#7}$} &   & \makebox[0pt][c]{${#7}$} \\
{\phantom{\makebox[0pt][c]{${#8}$}}} &   & . &   & . &   & \makebox[0pt][c]{${#6}$} &   & . &   & . &   \\
    &   &   &   &   &   & \ar@{<.>}[uuuuuuuuuu]  &   &   &   &   & \\
  } %
}

\newcommand{\kernelpic}[8]{%
  \xymatrix@R=0.5pc@C=0.7pc{
  &  &  &  &  &  & \makebox[0pt][c]{${#2}$} &  &  &  &  &  \\
{\phantom{\makebox[0pt][c]{${#8}$}}} &   & . &   & . &   & . &   & . &   & . &   \\
   & . &   & . &   & . &   & . &   & . &   & . \\
 &    & \makebox[0pt][c]{$#4$} &   & \makebox[0pt][c]{$#4$} &   & \makebox[0pt][c]{${#3}$} &   & \makebox[0pt][c]{$#5$} &   & \makebox[0pt][c]{$#5$} &   &  & \\
{\phantom{\makebox[0pt][c]{${#8}$}}} & . &   & . &   & . &   & . &   & \makebox[0pt][c]{${#7}$} &   & \makebox[0pt][c]{${#7}$} \\
{\phantom{\makebox[0pt][c]{${#8}$}}} &   & . &   & . &   & . &   & . &   & . &   \\
{\phantom{\makebox[0pt][c]{${#8}$}}} & . &   & . &   & . &   & . &   & \makebox[0pt][c]{${#7}$} &   & \makebox[0pt][c]{${#7}$} \\
   \ar@{<.>}[rrrrrrrrrrrr] {\phantom{\makebox[0pt][c]{${#8}$}}} &   & . &   & . &   & . &   & . &   & . &  & & \makebox[0pt][c]{${#1}$} \\
{\phantom{\makebox[0pt][c]{${#8}$}}} & . &   & . &   & . &   & . &   & . &   & . \\
{\phantom{\makebox[0pt][c]{${#8}$}}} &   & . &   & . &   & . &   & . &   & . &   \\
    &   &   &   &   &   & \ar@{<.>}[uuuuuuuuuu]  &   &   &   &   & \\
  } %
}

\newcommand{\covproj}[1]{P_{\!#1}} 
\newcommand{\contraproj}[1]{P^{\!#1}} 
\newcommand{\covinj}[1]{I_{\!#1}} 
\newcommand{\contrainj}[1]{I^{\!#1}} 

\newcommand{\gset}[1]{\mathbf{#1}}

\newcommand{\defword}[1]{\textbf{#1}}
\newcommand{\cxuniverse}{{\sU_{\bC}}}
\newcommand{\runiverse}{\sU_{\bR}}
\newcommand{\triv}{\{e\}}

\DeclareMathOperator{\proj}{proj}

\renewcommand{\to}{\longrightarrow}

\newtheorem{thm}{Theorem}[section]

\newaliascnt{cor}{thm}
\newtheorem{cor}[cor]{Corollary}
\aliascntresetthe{cor}

\newaliascnt{prop}{thm}
\newtheorem{prop}[prop]{Proposition}
\aliascntresetthe{prop}

\newaliascnt{lem}{thm}
\newtheorem{lem}[lem]{Lemma}
\aliascntresetthe{lem}

\newaliascnt{conj}{thm}

\aliascntresetthe{conj}

\newaliascnt{fact}{thm}

\aliascntresetthe{fact}

\newaliascnt{axiom}{thm}

\aliascntresetthe{axiom}

\newaliascnt{axioms}{thm}

\aliascntresetthe{axioms}

\theoremstyle{definition}

\newaliascnt{defn}{thm}
\newtheorem{defn}[defn]{Definition}
\aliascntresetthe{defn}

\newaliascnt{exmp}{thm}
\newtheorem{exmp}[exmp]{Example} 
\aliascntresetthe{exmp}

\theoremstyle{remark}

\newaliascnt{rem}{thm}
\newtheorem{rem}[rem]{Remark}
\aliascntresetthe{rem}

\makeatletter
\let\c@equation\c@thm
\makeatother
\numberwithin{equation}{section}

\title{Equivariant Spectral Sequences for Local Coefficients}

\author{Megan Guichard Shulman}

\date{\today}

\begin{document}

\begin{abstract}
We recall how a description of local coefficients that Eilenberg introduced
in the 1940s leads to spectral sequences for the computation of homology
and cohomology with local coefficients.  We then show how to construct new
equivariant analogues of these spectral sequences and give a worked example
of how to apply them in a computation involving the equivariant Serre
spectral sequence.

This paper contains some of the material in the author's Ph.D. thesis,
which also discusses the results of L.\ Gaunce Lewis \cite{lgl} on the
cohomology of complex projective spaces and corrects some flaws in that
paper.
\end{abstract}

\maketitle

\tableofcontents

Applications of the Serre spectral sequence in the literature usually use
trivial local coefficients. This perhaps reflects the tendency of much of
modern algebraic topology to steer away from the unstable world, where the
fundamental group cannot be ignored.  But it may also reflect our lack of
tools for computing the relevant homology and cohomology with local
coefficients.  Whatever the reason, it is impossible to ignore local
coefficients when working equivariantly with Bredon (co)homology and
the equivariant Serre spectral sequence; almost no interesting examples
reduce to trivial local coefficients.  It is thus necessary to develop some
tools for working with homology and cohomology with local coefficients.

To this end, we recall a simple universal coefficient spectral sequence
which appears in \cite[p.~355]{CE}, but which, to the best of our
knowledge, has not previously been applied in conjunction with the Serre
spectral sequence.  Part of the point is that the definition of local
coefficients that appears in the construction of the Serre spectral
sequence is not tautologically the same as the definition that gives the
cited universal coefficient spectral sequence. 

The connection comes from an old result of Eilenberg \cite{Eil},
popularized by Whitehead \cite[VI.3.4]{GW} and, more recently, by Hatcher
\cite[App3.H]{Hat}.  It identifies the local coefficients that appear in
the context of fibrations with a more elementary definition in terms of the
chains of the universal cover of the base space.  The identification makes
working with local coefficients much more feasible.  We shall first say
in \autoref{sec::noneq} how this goes nonequivariantly, and illustrate with
an example.  We then explain the equivariant generalization in
\autoref{sec::eq}. The later parts of the paper will develop some necessary
background and finally go through a sample calculation in
\autoref{sec::computations}.  

\section{The nonequivariant situation}\label{sec::noneq}

Let $X$ be a path-connected based space with universal cover $\tilde X$.
Let $\pi = \pi_1(X)$ and let $\pi$ act on the right of $\tilde{X}$ by deck
transformations.  Fix a ring $R$, and let $M$ be a left and $N$ be a right
module over the group ring $R[\pi]$.  Let $C_*$ be the normalized singular
chain complex functor with coefficients in $R$.

\begin{defn}\label{def::stdloc}  Define the homology of $X$ with coefficients in
$M$ by 
\[  H_*(X;M) = H_*(C_*(\tilde X)\otimes_{R[\pi]} M). \]
Define the cohomology of $X$ with coefficients in $N$ by \ \ 
\[ H^*(X;N) = H^*(\Hom_{R[\pi]}(C_*(\tilde X),N)). \]
\end{defn}

Functoriality in $M$ and $N$ for fixed $X$ is clear.  For a based map
$f\colon X\to Y$, where $\pi_1(Y) = \rho$, and for a left $R[\rho]$-module
$P$, we may regard $P$ as a $R[\pi]$-module by pullback along $\pi_1(f)$,
and then, using the standard functorial construction of the universal
cover, we obtain
\[ f_*\colon H_*(X;f^*P)\to H_*(Y;P). \]
Cohomological functoriality is similar.  The definition goes back to
Eilenberg \cite{Eil}, and has the homology of spaces and the homology of
groups as special cases, as discussed below.  It deserves more emphasis
than it is usually given because it implies spectral sequences for the
calculation of homology and cohomology with local coefficients, as we shall
recall. 

\begin{exmp}  If $\pi$ acts trivially on $M$ and $N$, then $H_*(X;M)$ and
$H^*(X;N)$ are the usual homology and cohomology groups of $X$ with
coefficients in $M$ and $N$.  We can identify $C_*(X)$ with
$C_*(\tilde{X})\otimes_{R[\pi]}R$, where $R[\pi]$ acts trivially on $R$.
This implies the identifications 
\[C_*(\tilde X)\otimes_{R[\pi]} M\cong C_*(X;M)\hspace{12pt}
\text{and}\hspace{12pt} \Hom_{R[\pi]}(C_*(\tilde X),N) \cong C^*(X;N). \]
\end{exmp}

\begin{exmp}  If $X = K(\pi,1)$, then $H_*(X;M)$ and $H^*(X;N)$ are the
usual homology and cohomology groups of $\pi$ with coefficients
in $M$ and $N$ since $C_*(\tilde{X})$ is an $R[\pi]$-free resolution of $R$.
That is,
\[  H_*(K(\pi,1);M) = \Tor^{R[\pi]}_*(R,M) \hspace{12pt} 
\text{and}\hspace{12pt} H^*(K(\pi,1);M) = \Ext_{R[\pi]}^*(R,N). \]
\end{exmp}

\begin{exmp}  If $M =R[\pi]\otimes_R A$ and $N = \Hom_R(R[\pi],A)$ for an
$R$-module $A$, then \[H_*(X;M)\cong H_*(\tilde{X},A)\hspace{12pt}
\text{and}\hspace{12pt} H^*(X;N) \cong H^*(\tilde{X};A). \]
\end{exmp}

\begin{rem}  If we replace $N$ by $M$ 
(viewed as a right $R[\pi]$-module) in the cohomology 
case of the previous example, then we are forced to 
impose finiteness restrictions and consider cohomology
with compact supports; compare \cite[3H.5]{Hat}.
\end{rem}

We have spectral sequences that generalize the last two examples.  When
$\pi$ acts trivially on $M$ and $N$, they can be thought of as versions of
the Serre spectral sequence of the evident fibration $\tilde{X}\to
X\to K(\pi,1)$. 

\begin{thm}[Eilenberg Spectral Sequence]\label{thm::CE}  There are spectral sequences
\[  E^2_{p,q} = \Tor^{R[\pi]}_{p,q}(H_*(\tilde{X}),M) \Longrightarrow  H_{p+q}(X;M)  \]
and
\[  E_2^{p,q} = \Ext_{R[\pi]}^{p,q}(H_*(\tilde{X}),N) \Longrightarrow  H^{p+q}(X;N).  \]
\end{thm}

Up to notation, these are the spectral sequences given by Cartan and
Eilenberg in \cite[p.~355]{CE}.

\begin{proof}  In the $E^2$ and $E_2$ terms, $p$ is the homological degree
and $q$ is the internal grading on $H_*(\tilde{X})$.  Let $\varepsilon\colon
P_*\to M$ be an $R[\pi]$-projective resolution of $M$ and form the
bicomplex
\[ C_*(\tilde{X})\otimes_{R[\pi]} P_*;\]
the theorem comes from looking at the two spectral sequences associated to
this bicomplex and converging to a common target.

If we filter $C_*(\tilde{X})\otimes_{R[\pi]} P_*$ by the degrees of
$C_*(\tilde{X})$, we get a spectral sequence whose $E^0$-term has
differential $\id\otimes d$.  Since $C_*(\tilde{X})$ is a projective
$R[\pi]$ module, the resulting $E^1$-term is
$C_*(\tilde{X})\otimes_{R[\pi]}M$, the resulting $E^2$-term is $H_*(X;M)$,
and $E^2=E^{\infty}$.  Since $E^\infty$ is concentrated in degree $q=0$,
there is no extension problem; we have identified the target as claimed in
the theorem.

Filtering the other way, by the degrees of $P_*$, we obtain a spectral
sequence whose $E^0$-term has differential $d\otimes \id$.  The resulting
$E^1$-term is $H_*(\tilde{X})\otimes_{R[\pi]}P_*$ and the resulting
$E^2$-term is $\Tor^{R[\pi]}_{*,*}(H_*(\tilde{X}),M)$.  This gives the
first statement of the theorem.

The argument in cohomology is similar, starting from the bicomplex
\[\Hom_{R[\pi]}(C_*(\tilde{X}),I^*)\]
for an injective resolution $\eta\colon N\to I^*$ of $N$.
\end{proof}

We record an immediate corollary.

\begin{cor}\label{cor::CE}
Let $\pi$ be a finite group of order $n$ and $R$ be a field of
characteristic prime to $n$.  Then
\[ H_*(X;M) \cong H_*(\tilde{X})\otimes_{R[\pi]}M
\hspace{12pt} \text{and}\hspace{12pt} 
H^*(X;N) \cong \Hom_{R[\pi]}\left(H_*(\tilde{X}),N\right). \]
\end{cor}

\begin{proof}
Since $R[\pi]$ is semi-simple, $E_{p,q}^2 = 0$ and $E_2^{p,q} = 0$ for
$p>0$.  Therefore the spectral sequences collapse to the claimed
isomorphisms.
\end{proof}

If $\pi$ acts trivially on $H_*(\tilde{X})$, so that $H_*(\tilde{X}) \cong
H_*(\tilde{X})\otimes_{R[\pi]} R$, then the situation simplifies even
further.  For ease of notation, let $M_\pi$ denote the coinvariants
$M/{IM}$, where $I\subset  R[\pi]$ is the augmentation ideal, and let
$N^\pi$ denote the fixed points of $N$.

\begin{cor}\label{cor::cor::CE}
Suppose that we are in the situation of \autoref{cor::CE} and that $\pi$
acts trivially on $H_*(\tilde{X})$.  Then
\[ H_*(X;M) \cong H_*(\tilde{X};M_\pi) \hspace{12pt} \text{and}\hspace{12pt}
H^*(X;N) \cong H^*(\tilde{X};N^\pi).\qedhere\]
\end{cor}
\vspace{-4.5ex}

\hfill$\Box$
\vspace{2.5ex}

It remains to identify the homology and cohomology groups of
\autoref{def::stdloc} with classical (co)homology with local coefficients.
To do this, we first need to reconcile our coefficient $R[\pi]$-modules
with the classical definition of a local coefficient system.

In \autoref{def::stdloc}, we took $M$ and $N$ to be left and right modules
over the group ring $R[\pi]$ and took $C_*(\tilde{X})$ to be the normalized
singular chains of $\tilde{X}$.  A (left or right) $R[\pi]$-module $M$ is
the same as a (covariant or contravariant) functor from $\pi$, viewed as a
category with a single object, to the category of $R$-modules.  

As usual, given a space $X$, let $\Pi X$ be the fundamental groupoid of
$X$; this is a category whose objects are the points of $X$ and whose
morphism sets are homotopy classes of paths between fixed endpoints.
By definition, a \defword{local coefficient system} $\sM$ on $X$ is
a functor (covariant or contravariant depending on context, corresponding
to our left and right $R[\pi]$-module distinction above) from the
fundamental groupoid $\Pi X$ to the category of $R$-modules.   When
$X$ is path-connected with basepoint $x_0$, the category $\pi =\pi_1(X)$
with single object $x_0$ is a skeleton of $\Pi X$.  Therefore a coefficient
system $\sM$ is determined by its restriction $M$ to $\pi$.  

Whitehead \cite[VI.3.4 and 3.4*]{GW} (see
also Hatcher \cite[3H.4]{Hat}) proves the following result and ascribes it 
to Eilenberg \cite{Eil}.  

\begin{thm}[Eilenberg]\label{thm::comparison}
For path-connected spaces $X$ and covariant and contravariant local
coefficient systems $\sM$ and $\sN$ on $X$, the classical homology
and cohomology with local coefficients $H_*(X;\sM)$ and $H^*(X;\sN)$
are naturally isomorphic to the homology and cohomology groups $H_*(X;M)$
and $H^*(X;N)$, where $M$ and $N$ are the restrictions of $\sM$ and $\sN$
to $\pi$.  \end{thm}

Therefore \autoref{thm::CE} gives a way to compute the additive structure of
homology and cohomology with local coefficients.  In particular, if
$f\colon E\to X$ is a fibration with fiber $F$ and path-connected base
space $X$, it gives a means to compute the homology and cohomology with
local coefficients that appear in
\[ E_{*,*}^2 = H_*(X;\sH_*(F;R))\hspace{12pt} \text{and}\hspace{12pt} 
E_2^{*,*} = H^*(X;\sH^*(F;R)) \]
of the Serre spectral sequences for the computation of $H_*(E ; R)$ and
$H^*(E ; R)$.

Even the case when $\pi$ is finite of order $n$ and $R$ is a field of
characteristic prime to $n$ often occurs in practice.  More generally, the
spectral sequences of \autoref{thm::CE} help make the Serre spectral
sequence amenable to explicit calculation in the presence of non-trivial
local coefficient systems. 

Note that we have not yet addressed the multiplicative structure of the
cohomological Eilenberg spectral sequence; this is work in progress and is
discussed in the author's Ph.D. thesis.  However, even without the
multiplicative structure, we can already do one example.  It will be useful
to have the following simple consequence of \autoref{def::stdloc}.

\begin{prop}\label{prop::hzero}
Let $X$ be a space and $\pi$ its fundamental group.  For any
$R[\pi]$-module $M$, $H^0(X;M) \cong M^\pi$, the $\pi$-fixed points of $M$.
If $M$ is an $R[\pi]$-algebra, the isomorphism is as algebras.
\end{prop}

\begin{proof}
We would like to identify the kernel of
\[ \Hom_{R[\pi]}(C_0(\tilde{X}),M) \to \Hom_{R[\pi]}(C_1(\tilde{X}),M).\]
Since $\tilde{X}$ is connected, $H_0(\tilde{X}) \cong R$ with the trivial
$\pi$ action.  Since $\Hom_{R[\pi]}$ is left exact, it follows that
$H^0(X;M) \cong \Hom_{R[\pi]}(R,M) \cong M^\pi$, as claimed.
\end{proof}

In particular, in the Serre spectral sequence, $E_2^{0,t}
\cong \sH^t(F;R)^\pi$.  If we are in a situation where $E_2^{s,t}$ vanishes
for $s>0$, then \autoref{prop::hzero} completely describes the
multiplicative structure on the $E_2$ page.

\begin{exmp}\label{ex::noneq}
Let $\cyctwo$ be the cyclic group of order two, identified with $\{\pm 1\}$
when convenient.  Consider the fibration 
\[B\text{det}\colon BO(2)\to B\cyctwo\] 
with fiber $BSO(2)$.  Take coefficients in a finite field $R = \fq$ with
$q$ an odd prime, so that $R[\pi]$ is semisimple.  The action of the base
on the fiber is nontrivial; the Serre spectral sequence for this fibration
has
\[ E_2^{s,t} = H^s(B\cyctwo;\sH^t(BSO(2);\fq)) \Longrightarrow H^{s+t}(BO(2);\fq)
\]
We know that $B\cyctwo \simeq \bR P^\infty$ with universal cover $S^\infty
\simeq \bullet$, so \autoref{cor::cor::CE} applies, and $BSO(2) \simeq \bC
P^\infty$.  We also know $H^*(\bC P^\infty;\fq) \cong \fq[x]$, a polynomial
algebra on one generator $x$ in degree two, and that the fundamental group
$\pi_1(B\cyctwo) \cong \cyctwo$ acts on $H^*(\bC P^\infty)$ by $x\mapsto
-x$. 

By \autoref{cor::cor::CE}, we thus have
\[ E_2^{s,t} \cong H^s \left(S^\infty;\sH^t(BSO(2);\fq)^{\cyctwo} \right) \]
in the Serre spectral sequence; $\sH^t(BSO(2);\fq)^{\cyctwo}$ is either $0$
or $\fq$, depending on $t$.  $E_2^{s,t}$ thus vanishes for $s>0$, so the Serre
spectral sequence collapses with no extension problems.  Using the
observation after \autoref{prop::hzero}, we see
\[ H^*(BO(2);\fq) \cong \sH^*(BSO(2);\fq)^{\cyctwo} \cong \fq[x^2].\]
The isomorphism is of $R$-algebras.  We have thus shown the well-known fact
that $H^*(BO(2);\fq)$ is polynomial on one generator (the Pontrjagin class)
in degree four.  We will later examine an equivariant version of
this example.
\end{exmp}

\section{Equivariant generalizations}\label{sec::eq}

Heading towards an equivariant generalization of \autoref{thm::CE}, we first
rephrase the definition of (co)homology with local coefficients. 
In \autoref{sec::noneq}, we effectively defined homology and cohomology
with local coefficients by restricting a local coefficient system
$\sM\colon \Pi X \to \rmod$ to an $R[\pi]$-module $M\colon \pi \to \rmod$.

Rather than restricting $\sM$ to $\pi$, we could instead redefine
$\tilde{X}$ to be the universal cover functor $\Pi X\op \to \spaces$ that
sends a point $x\in X$ to the space $\tilde{X}(x)$ of equivalence classes
of paths starting at $x$ and sends a path $\gamma$ from $x$ to $y$ to the
map $\tilde{X}(y)\to \tilde{X}(x)$ given by precomposition with $\gamma$.
Since $\pi$ is a skeleton of $\Pi X$, the following definition is
equivalent to \autoref{def::stdloc} when $X$ is connected.  By
\autoref{thm::comparison}, there is no conflict with the classical notation
for homology with local coefficients.  Let $\chr$ denote the category of
chain complexes of $R$-modules.

\begin{defn}[Reformulation of \autoref{def::stdloc}]\label{def::altloc} 
Let $\sM\colon \Pi X\to \rmod$ and $\sN\colon \Pi X\op\to \rmod$ be
functors and let $C_*(\tilde{X})\colon \Pi X\op\to \spaces \to \chr$ be the
composite of the universal cover functor with the functor $C_*$.  Define
the homology of $X$ with coefficients in $\sM$ to be \[ H_*(X;\sM) =
H_*(C_*(\tilde{X}) \otimes_{\Pi X} \sM) \] where $\otimes_{\Pi X}$ is the
tensor product of functors (which is given by an evident coequalizer
diagram).  Similarly, define 
\[ H^*(X;\sN) = H^*\left(\Hom_{\Pi X}(C_*(\tilde{X}),\sN)\right) \]
where $\Hom_{\Pi X}$ is the hom of functors (also known as natural
transformations; alternatively, given by an evident equalizer diagram).
\end{defn}


Note that our distinctions between left and right and between covariant and
contravariant are purely semantic above, since we are dealing with groups
and groupoids.  However, we are about to consider (Bredon) equivariant
homology and cohomology. Here the fundamental ``groupoid'' is only an
EI-category (endomorphisms are isomorphisms) and the distinction is
essential. There is an equivariant Serre spectral sequence, due to Moerdijk
and Svensson \cite{MS}, but it has not yet had significant calculational
applications. The essential reason is the lack of a way to compute its
$E^2$-terms.  However, the results of \autoref{sec::noneq} generalize nicely
to compute Bredon homology and cohomology with local coefficients.

\autoref{def::altloc} generalizes directly to the equivariant case.  From now on,
let $X$ be a $G$-space, where $G$ is a discrete group.\footnote{With a
little more detail, we could generalize to topological groups.}  Following
tom Dieck \cite{tD}, we can define the \defword{fundamental EI-category}
$\fund{X}$ to be the category whose objects are pairs $(H,x)$, where $x \in
X^H$; a morphism from $(H,x)$ to $(K,y)$ consists of a $G$-map
$\alpha\colon G/H\to G/K$, determined by $\alpha(eH) = gK$, together with a
homotopy class rel endpoints $[\gamma]$ of paths from $x$ to $\alpha^*(y) =
gy$.  Here $\alpha^*\colon X^K\to  X^H$ is the map given by $\alpha^*(z)=gz$,
which makes sense because $g^{-1}Hg\subset K$.

Likewise, we follow tom Dieck in defining the equivariant universal cover
$\tilde{X}$ to be the functor $\tilde{X}\colon (\fund{X})\op \to \spaces$
which sends $(H,x)$ to $\widetilde{X^H}(x)$, the space of equivalence
classes of paths in $X^H$ starting at $x$.  For a morphism $\fmap{}\colon
(H,x)\to (K,y)$, $\tilde{X}\fmap{}\colon \tilde{X}(K,y)\to \tilde{X}(H,x)$
takes a class of paths $[\beta]$ starting at $y\in X^K$ to the class of the
composite $(\alpha^*\beta)\gamma$.

We can now define equivariant (co)homology with local coefficients.  In
fact, \autoref{def::altloc} applies almost verbatim: we need only add $G$
to the notations.  We repeat the definition for emphasis. 

\begin{defn}[Equivariant generalization of \autoref{def::altloc}]\label{def::eqloc}
Let $X$ be a $G$-space and write $\Pi = \fund{X}$.  Let $\sM\colon \Pi\to
\rmod$ and $\sN\colon \Pi\op\to \rmod$ be functors and let
$\eqchains{\tilde{X}}\colon \Pi\op\to \spaces \to \chr$ be the composite of
the equivariant universal cover functor with the functor $C_*$.  Define the
homology of $X$ with coefficients in $\sM$ to be
\[ H_*^G(X;\sM) = H_*(\eqchains{\tilde{X}} \otimes_\Pi \sM) \]
and the cohomology of $X$ with coefficients in $\sN$ by
\[ H^*_G(X;\sN) = H^*\left(\Hom_\Pi(C^G_*(\tilde{X}),\sN)\right). \]
\end{defn}

Note that we could also take $\Pi$ to be a skeleton
$\text{skel}(\fund{X})$. 

Inserting $G$ into the notations, the proofs in Whitehead or Hatcher
\cite{GW, Hat} apply to show that this definition of Bredon (co)homology
with local coefficients is naturally isomorphic to the Bredon (co)homology
with local coefficients, as defined in Mukherjee and Pandey \cite{MP},
which they in turn show is naturally isomorphic to the (co)homology with
local coefficients, as defined and used by Moerdijk and Svensson in
\cite{MS} to construct the equivariant Serre spectral sequence of a
$G$-fibration $f\colon E\to B$.

We quickly review the homological algebra needed for the equivariant
generalization of \autoref{thm::CE}. Since $\rmod$ is an Abelian category, the
categories $[\fundx{},\rmod]$ and $[\fundx{}\op,\rmod]$ of functors from
$\fundx{}$ to $\rmod$ are also Abelian, with kernels and cokernels defined
levelwise.  These categories have enough projectives and injectives, which
by the Yoneda lemma are related to the represented functors.

Specifically, let $R-$ denote the free $R$-module functor
$\textbf{Set}\to\rmod$.  Given an object $(H,x)\in \Pi$, let
$\covproj{H,x}$ be the covariant represented functor $\Pi\to\rmod$
given on objects by
\[ \covproj{H,x}(K,y) = R\fundhom{(H,x)}{(K,y)}.\]
By the Yoneda lemma, each $\covproj{H,x}$ is projective.  Therefore, given
a functor $\sM$, we can construct an epimorphism $\sP\to\sM$ with $\sP$
projective by taking $\sP$ to be a direct sum of representables
\[ \sP = \bigoplus_{(H,x)}\bigoplus_{\sM(H,x)}\covproj{H,x},\]
one for each element of each $R$-module $\sM(H,x)$.  Similarly, there are
contravariant represented functors $\contraproj{H,x}\colon \Pi\op\to\rmod$
given by
\[ \contraproj{H,x}(K,y) = R\fundhom{(K,y)}{(H,x)}. \]
The same argument shows that these are projective and that $[\Pi\op,\rmod]$
has enough projectives.

The construction of the injective objects is dual but perhaps less
familiar.  Given an $R$-module $A$ and $(H,x)\in\Pi$, we define a functor
$\covinj{H,x,A}\colon \Pi\to\rmod$ by
\[ \covinj{H,x,A}(K,y) = \Hom_R(\contraproj{H,x}(K,y),A) \]
Whenever $A$ is an injective $R$-module, $\covinj{H,x,A}$ is an injective
object in $[\Pi,\rmod]$.  This comes from a more general fact.  
For any coefficient system $\sA\colon \Pi\to\rmod$, there is a tensor-hom
adjunction
\[ [\Pi,\rmod](\sA,\covinj{H,x,A}) \cong \rmod(\sA\otimes_\Pi
\contraproj{H,x},A) \]
where again $\otimes_\Pi$ is the tensor product of functors.  The tensor
product of any functor with a representable functor $\contraproj{H,x}$ is
given by evaluation at $(H,x)$.  Putting these two facts together, we have
that a natural transformation from $\sA$ to $\covinj{H,x,A}$ is given by
the same data as a homomorphism of $R$-modules from $\sA(H,x)$ to $A$.  It
is then clear that, if $A$ is an injective $R$-module, $\covinj{H,x,A}$
must be an injective object of $[\Pi,\rmod]$, as desired.  Given any
$\sN\colon \Pi\to\rmod$, we can construct an injective coefficient system
$\sI$ and a monomorphism $\sN\to\sI$ as follows.  Choose monomorphisms
$\sN(H,x) \hookrightarrow A_{H,x}$ for each $(H,x)$ with $A_{H,x}$
injective, and define $\sI$ to be the product of injective functors
\[ \sI = \prod_{(H,x)} \covinj{H,x,A_{H,x}}. \]
It can be checked that the evident map $\sN\to\sI$ is a monomorphism.
Thus $[\Pi,\rmod]$ has enough injectives.  The functors 
$\contrainj{H,x,A} = \Hom_R(\covproj{H,x}(-),A)$ show that $[\Pi\op,\rmod]$
has enough injectives as well.

Finally, we define $\Tor^\fundx{}(\sN,\sM)$ in the obvious way.  It is the
homology of the complex of $R$-modules that is obtained by taking the
tensor product of functors of $\sN$ with a projective resolution of the
functor $\sM$.  We define $\Ext_\fundx{}(\sN_1,\sN_2)$ similarly, taking
the hom of functors of $\sN_1$ with an injective resolution of
$\sN_2$.\footnote{Alternatively, we could define $\Ext_{\fundx{}}(\sN_1,\sN_2)$
by taking a projective resolution of $\sN_1$; however, it is the definition
above which gives rise to the appropriate spectral sequence.}

The following equivariant analogue of the nonequivariant statement that
$C_*(\tilde{X})$ is a free $R[\pi]$-module should be a standard first
observation in equivariant homology theory, but the author has not seen it
in the literature.  The nonequivariant assertion, while obvious, is the
crux of the proof of \autoref{thm::CE}.  Let $\contraproj{H,x} =
R\fundhom{-}{(H,x)}$, as above.

\begin{lem}\label{lem::technical}
With $\Pi = \fund{X}$, each functor $C_n^G(\tilde{X})\colon \Pi\op \to
\rmod$ is a direct sum of representable functors
$\bigoplus_{(H_i,x_i)}\contraproj{H_i,x_i}$. 
\end{lem}

Granting this result for the moment, we can prove the equivariant
generalization of \autoref{thm::CE}.

\begin{thm}[Equivariant Eilenberg Spectral Sequence]\label{thm::eqCE}
With $\Pi = \fund{X}$, there are spectral sequences
\[ E_{p,q}^2 = \Tor_{p,q}^\fundx{}(\coh_*(\tilde{X}),\sM) \Longrightarrow
H_{p+q}^G(X;\sM) \]
and
\[  E_2^{p,q} = \Ext_{\fundx{}}^{p,q}(\coh_*(\tilde{X}),\sN) \Longrightarrow  H^{p+q}_G(X;\sN).  \]
\end{thm}

Here the functor $\coh_*(\tilde{X})\colon \fundx{}\op\to \rmod$ is the
homology of the chain complex functor $\eqchains{\tilde{X}}$; that is,
$\coh_*(\tilde{X})(H,x)$ is the homology of the chain complex
$C_*(\tilde{X})(H,x)$.

\begin{proof}
Let $\varepsilon\colon \sP_*\to\sM$ be a projective resolution of $\sM$.  As in
the nonequivariant theorem, form the bicomplex of $R$-modules
$C_*(\tilde{X}) \otimes_\fundx{} \sP_*$.  Since the tensor product of a
functor with a representable functor is given by evaluation,
\[ \contraproj{H,x} \otimes_\fundx{} \sM \cong \sM(H,x),\]
tensoring with such projective modules is exact.

In particular, if we filter our bicomplex by degrees of $C_*(\tilde{X})$,
then $d^0 = \id\otimes d$.  By \autoref{lem::technical}, each
$C_n^G(\tilde{X})$ is projective, and so we get a spectral sequence with
$E^1$-term $C_*(\tilde{X})\otimes_\fundx{} \sM$.  Thus the resulting $E^2 =
E^\infty$ term is $H_*(X,\sM)$, exactly as in the nonequivariant case.

If we instead filter by degrees of $\sP_*$, so $d^0 = d\otimes\id$, then
the $E^1$ term is $\coh_*(\tilde{X})\otimes_\fundx{}\sP_*$ and the $E^2$
term is $\Tor_{*,*}^\fundx{}(\coh_*(\tilde{X}),\sM)$, as desired.

The construction of the second spectral sequence is similar, starting from
an injective resolution $\eta\colon \sN \to \sI^*$.
\end{proof}

\begin{proof}[Proof of \autoref{lem::technical}]
The proof is analogous to that of the nonequivariant result, but more
involved.  We may identify $C_n^G(\tilde{X})(H,x)$ with the free $R$-module
on generators given by the nondegenerate singular $n$-simplices
$\sigma\colon \Delta^n\to \tilde{X}(H,x)$.  We must show that these free
$R$-modules piece together appropriately into a free functor.  More
specifically, by the Yoneda lemma, each $\sigma\colon \Delta^n\to
\tilde{X}(K,y)$ determines a natural transformation
\[\iota_\sigma\colon \contraproj{K,y}\to C_n^G(\tilde{X})  \]
that takes $\id \in \fundhom{(K,y)}{(K,y)}$ to $\sigma$.  We thus obtain a
natural transformation 
\[ \bigoplus_{\{\tau\}} \contraproj{K_\tau,y_\tau} \to C_n^G(\tilde{X}) \]
from any set of nondegenerate $n$-simplices $\{\tau\colon \Delta^n\to
\tilde{X}(K_\tau,y_\tau)\}$. We must show that there is a set $\{\tau\}$
such that the resulting natural transformation is a natural isomorphism,
that is, a levelwise isomorphism.  This amounts to showing that the
following statements hold for our choice of generators $\tau$ and each
object $(H,x)$.  
\begin{enumerate}
\item (Injectivity) For any arrows $\fmap{1}$ and $\fmap{2}$ in $\fundx{}$
with source $(H,x)$ and any generators $\tau_{1}$ and $\tau_{2}$,
$\fmap{1}^*\tau_{1} = \fmap{2}^*\tau_{2}$ must imply that both $\fmap{1} =
\fmap{2}$ and $\tau_{1} = \tau_{2}$.  

\item (Surjectivity) For every $\sigma\colon \Delta^n\to \tilde{X}(H,x)$,
there must be a generator $\tau$ and an arrow $\fmap{}$ such that $\sigma =
\fmap{}^* \tau$.
\end{enumerate}

Fixing $n$, define the generating set as follows.  Regard the initial
vertex $v$ of $\Delta^n$ as a basepoint.  Recall that $\tilde{X}(K,y)$ is
the universal cover of $X^K$ defined with respect to the basepoint $y\in
X^K$, so that the equivalence class of the constant path $c_{K,y}$ at $y$
is the basepoint of $\widetilde{X^K}$.  In choosing our generating set, we
restrict attention to based maps $\sigma\colon \Delta^n\to
\tilde{X}(K_\sigma,y_\sigma)$ that are non-degenerate $n$-simplices of
$\widetilde{X^{K_\sigma}}$.  Such maps $\sigma$ are in bijective
correspondence with based nondegenerate $n$-simplices $\sigma_0\colon
\Delta^n\to X^{K_\sigma}$.  The correspondence sends $\sigma$ to its
composite with the end-point evaluation map $p\colon
\tilde{X}(K_\sigma,y_\sigma)\to X^{K_\sigma}$ and sends $\sigma_0$ to the
map $\sigma\colon \Delta^n\to \tilde{X}(K_\sigma,y_\sigma)$ that sends a
point $a\in \Delta^n$ to the image under $\sigma_0$ of the straight-line
path from $v$ to $a$.  Restrict further to those $\sigma$ that cannot be
written as a composite 
\[ \xymatrix@1{ \Delta^n \ar[r]^-{\rho} & \tilde{X}(K',y')
\ar[rr]^-{(\alpha,\gamma)^*} & & \tilde{X}(K_\sigma,y_\sigma)\\} \]
for any non-isomorphism $(\alpha,\gamma)\colon (K',y')\to (K_\sigma,y_\sigma)$ in
$\Pi$.  Note that, for each such $\sigma$, we can obtain another such
$\sigma$ by composing with the isomorphism $(\xi,\delta)^*$ induced by an
isomorphism $(\xi,\delta)$ in $\Pi$.  We say that the resulting maps $\sigma$
are equivalent, and we choose one $\tau$ in each equivalence class of such
based singular $n$-simplices $\sigma$.

It remains to verify that the natural transformation defined by this set
$\{\tau\}$ is an isomorphism.  This is straightforward but somewhat tedious
and technical.

For the injectivity, suppose that $\fmap{1}^* \tau_{1} = \fmap{2}^*
\tau_{2}$, where $\tau_1,\tau_2$ are in our generating set and 
\begin{align*}
\tau_1\colon \Delta^n\to \tilde{X}(K_1,y_1),  & \hspace{12pt}   \fmap{1} \in \fundhom{(H,x)}{(K_1,y_1)} \\
\tau_2\colon \Delta^n\to \tilde{X}(K_2,y_2),  & \hspace{12pt} \fmap{2} \in \fundhom{(H,x)}{(K_2,y_2)}.
\end{align*}
Since $\tau_i(v) = c_{(K_i,y_i)}$ for $i = 1,2$, we see that $\fmap{i}^*
\tau_i$ must take $v$ to $[\gamma_i]$.  Since $\fmap{1}^* \tau_{1} =
\fmap{2}^* \tau_2$, this means that $[\gamma_1] = [\gamma_2]$; call this
path class $[\gamma]$.  In turn, this implies that $\alpha_1^* y_1 =
\alpha_2^* y_2$; call this point $z\in X^H$, so that $[\gamma]$ is a path from
$x$ to $z$.  Since $\fmap{i} = (\alpha_i,[c_{z}])\circ (\id,[\gamma])$ and
$(\id,[\gamma])$ is an isomorphism in $\fundx{}$, we must have 
\[ (\alpha_1,[c_z])^* \tau_1 = (\alpha_2,[c_z])^* \tau_2. \]
In particular, if we compose each side of this equation with $p$, we obtain 
\[ p \circ (\alpha_1,[c_z])^* \tau_1 = p\circ (\alpha_2,[c_z])^* \tau_2\]
as maps $\Delta^n \to X^H$.  Since we have commutative diagrams 
\[ \xymatrix{ \Delta^n \ar[r]^-{\tau_i} & \tilde{X}(K_i,y_i)
\ar[r]^-{(\alpha,[c_z])^*} \ar[d]_{p} & \tilde{X}(H,z) \ar[d]_{p} \\ &
X^{K_i} \ar[r]^{\alpha_i^*} & X^H }\]
for each $i$, this implies that we have a commutative square 
\[ \xymatrix{ \Delta^n \ar[r]^-{p\circ\tau_1} \ar[d]_{p\circ\tau_2} &
X^{K_1} \ar[d]^{\alpha_1^*} \\ X^{K_2} \ar[r]^{\alpha_2^*} & X^H }\]
If the maps $\alpha_i \colon G/H \to G/{K_i}$ are defined by elements
$g_i\in G$, this implies that the common composite $\Delta^n \to X^H$
factors through the fixed-point sets $X^{g_i K_i g_i^{-1}}$ for each $i$,
and hence through $X^L$, where $L$ is the smallest subgroup containing $g_1
K_1g_1^{-1}$ and $g_2 K_2 g_2^{-1}$.  Since $K_i \subset g_i^{-1} L g_i$,
the maps $\alpha_i\colon G/H\to G/K_i$ factor through the maps $\beta_i\colon
G/K_i\to G/L$ specified by $\beta_i(eK_i) = g_iL$ and there result
factorizations of the $\tau_i$ as 
\[ \xymatrix@1{ \Delta^n \ar[r] &  \tilde{X}(g_i^{-1} L g_i, y_i)
\ar[rr]^-{(q,[c_z])^*} & &  \tilde{X}(K_i,y_i),\\}  \]
where $q$ denotes either quotient map $G/K_i\to G/g_i^{-1} L g_i$.  By
our choice of the generators $\tau$, this can only happen if $g_i^{-1} L
g_i = K_i$, giving $g_1 K_1 g_1^{-1} = g_2 K_2 g_2^{-1}$.  In terms of
$g_1$ and $g_2$, we see that our equation $(\alpha_1,[c_z])^* \tau_1 =
(\alpha_2,[c_z])^* \tau_2$ says that $g_1 \tau_1 = g_2 \tau_2$, that is,
$\tau_2 = g_2^{-1} g_1 \tau_1$.  Since $g_2^{-1} g_1$ defines an
isomorphism $G/{K_2}\to G/{K_1}$, we again see by our choice of the
generators $\tau$ that $\tau_1 = \tau_2$ and that $g_2^{-1} g_1 \in K_1 =
K_2$.  This in turn implies that the maps $\alpha_i \colon G/H \to G/K_i$
defined by the $g_i$ are identical.  The conclusion is that $\fmap{1}^*
\tau_1 = \fmap{2}^* \tau_2$ implies $\tau_1 = \tau_2$ and $\fmap{1} =
\fmap{2}$, as desired.

It only remains to show that we have accounted for all elements of
$C_n(\tilde{X})(H,x)$.  For any map $\sigma\colon \Delta^n \to
\tilde{X}(H,x)$, $\sigma(v)$ is a homotopy class of paths from $(H,x)$ to
$(H,x')$ in $X^H$.  Call this class $[\gamma]$.  Then $(\id,[\gamma])$ is
an isomorphism with inverse $(\id,[\gamma^{-1}])$, and $\sigma' =
(\id,[\gamma^{-1}])^*\sigma$ takes $v$ to the homotopy class of the
constant path at $(H,x')$; it follows that $\sigma =
(\id,[\gamma])^*\sigma'$.  Similarly, if $\sigma'$ factors through
$\tilde{X}(K,y)$ for some $K$ properly containing a conjugate of $H$, then
by definition $\sigma = \fmap{}^*\tau$ for some $\tau\colon \Delta^n \to
\tilde{X}(K,y)$.  We can choose a $\tau$ that does not itself factor and is
in our chosen set of generators.
\end{proof}

\section{Some remarks on the Serre Spectral Sequence}

We say that a map $f\colon E\to X$ is a \defword{$G$-fibration} if $f^H
\colon E^H \to X^H$ is a fibration for every subgroup $H$ of $G$.  Observe
that, if $x\in X^H$, then its preimage $f^{-1}(x)\subset E$ is necessarily
an $H$-space.  As previously mentioned, Moerdijk and Svensson in \cite{MS}
develop an equivariant Serre spectral sequence for $G$-fibrations of
$G$-spaces, which we will now describe.  

Given a coefficient system $M\colon \sO_G\op\to \rmod$ and a subgroup $H <
G$, there is a restricted coefficient system $M|_H \colon \sO_H\op \to \rmod$
given by $M|_H (H/K) := M(G/K)$.  We may thus define a local coefficient
system
\[\sh^q(f;M )\colon \fund{X}\op\to\ab \]
to be the functor which acts on objects by
\[ (H,x) \longmapsto H^q_G(G\times_H f^{-1}(x);M ) \cong H^q_H(f^{-1}(x);M |_H).\]
It is defined on morphisms via lifting of paths.

The main result of \cite{MS} is the following spectral sequence.

\begin{thm}[Moerdijk and Svensson]
For any $G$-fibration $f\colon E\to X$ and any coefficient system $M$,
there is a natural spectral sequence 
\[ E_2^{s,t}(M) = H^s_G(X;\sh^t(f;M)) \Longrightarrow \h^{s+t}(E;M).\]
Further, this spectral sequence carries a product structure, in the sense
that there is a natural pairing of spectral sequences
\[ E_r^{s,t}(M) \otimes E_r^{s',t'}(N) \to
E_r^{s+s',t+t'}(M\otimes N) \]
converging to the standard pairing
\[ H^*_G(E;M) \otimes H^*_G(E;N) \xrightarrow{\smile}
H^*_G(Y;M\otimes N). \]
On the $E_2$ page, this pairing agrees 
with the standard pairing
\[ H^s_G(X;\sh^t(f,M))\otimes \h^{s'}(X;\sh^{t'}(f,N)) \to
H^{s+s'}_G(X;\sh^{t+t'}(f,M\otimes N)).\footnote{As usual in the Serre
spectral sequence, the standard pairing on the $E_2$ page incorporates a
sign $(-1)^{s' t}$ relative to the cup product pairing.}\]
\end{thm}

The tensor product $M\otimes N$ above is a levelwise tensor product,
$(M\otimes N)(G/H) = M(G/H)\otimes N(G/H)$, which is distinct from the box
product $\boxp$ discussed in \autoref{sec::mackey}.

Although we have been using the integer-graded equivariant cohomology
originally defined by Bredon, equivariant cohomology is more naturally
graded on $RO(G)$.  As discussed in \cite[IX]{AK}, for any coefficient
system $M$ which can be extended to a Mackey functor, the Bredon cohomology
theory $H^*_G(-;M)$ can be extended to an $RO(G)$-graded theory.  That is,
for every virtual representation $\omega = V-W$, we have a functor
$H^{\omega}_G(-;M)$; these satisfy appropriate versions of the usual
axioms, including suspension: $\redh^{\omega+V}(\Sigma^V X;M) \cong
\redh^\omega(X;M)$ for any honest representation $V$.  One way of
visualizing this extra data is to say that we have one integer-graded
theory $\{H^{V+n}_G\}_{n\in\bZ}$ for each representation $V$ containing no
trivial subrepresentations.  Each of these theories $H^{V+*}_G(-;M)$ can be
used to define local coefficient systems $\sh^{V+t}(f,M)$.  Kronholm shows
the following in his thesis \cite{bill}.

\begin{thm}[Kronholm]\label{thm::bill}
For each real representation $V$, there is a natural spectral sequence
\[ E_2^{s,t}(M,V) = \h^s(X;\sh^{V+t}(f,M)) \Longrightarrow
\h^{V+s+t}(E;M). \]
Further, for each $V,V'\in RO(G)$, there is a pairing
\[ E_r^{s,t}(M,V) \otimes E_r^{s',t'}(N,V') \to
E_r^{s+s',t+t'}(M\otimes N,V+V')\]
converging to the standard pairing on $E_\infty$ and agreeing 
with the standard pairing on $E_2$.
\end{thm}

We have an analogue of \autoref{prop::hzero}, as well.

\begin{prop}\label{prop::eqhzero}
$\h^0(X;\sM) \cong \Hom_{\Pi}(\underline{R},\sM)$, where $\underline{R}$ is
the constant functor.
\end{prop}

\begin{proof}
As in \autoref{prop::hzero}, this comes from identifying $\coh_0(\tilde{X})
\cong \underline{R}$ and from the left exactness of $\Hom_\Pi$.
\end{proof}

\begin{prop}\label{prop::gconnected}
Suppose that $X$ is $G$-connected, in the sense that each $X^H$ is nonempty
and connected, and let $\bullet \in X^G$.  Then
$\Hom_{\Pi}(\underline{R},\sM)$ is isomorphic to a sub-$R$-module of
$\sM(G,\bullet)$.
\end{prop}

\begin{proof}
Since $X$ is $G$-connected, $(G,\bullet)$ is a weakly terminal object in
$\fund{X}$, i.e.~for every $(K,y)$ there is a map $(K,y)\to (G,\bullet)$.
It follows that an element of $\Hom_{\Pi}(\underline{R},\sM)$ is determined
by the map of $R$-modules $R\to \sM(G,\bullet)$.
\end{proof}

\section{Mackey functor valued cohomology theories}\label{sec::mackey}

Preparatory to the computation in \autoref{sec::computations}, we will
review the theory of Mackey functor valued cohomology theories in this
section, and the calculation by Lewis \cite{lgl} of $H^*_G(\cpv{V})$ in
\autoref{sec::cpv}.  Lewis considers cohomology which is not only
$RO(G)$-graded but Mackey functor valued.  He takes coefficients in the
Burnside ring Mackey functor $A = A_G$.  Fix a finite group $G$.

\begin{defn}\label{def::burnsidestable}
The \defword{Burnside category} $\sB_G$ is the full subcategory of the
equivariant stable category on objects $\Sigma^\infty (\gset{b}_+)$, where
$\gset{b}$ is a finite $G$-set.  Explicitly, the objects of $\sB_G$ are
finite $G$-sets and the morphisms are the stable $G$-maps.
\end{defn}

\begin{defn} A \defword{Mackey functor} is a contravariant additive functor
from the Burnside category $\sB_G$ to the category $\rmod$.
\end{defn}

For obvious reasons, $\sB_G$ is sometimes called the ``stable orbit
category'' and Mackey functors ``stable coefficient systems.''  However, it
is frequently more convenient to work with the following combinatorial
definition.  It is shown in \cite{AK} that the two definitions are
equivalent for finite $G$.

\begin{defn}\label{def::preburnside}
The category $\sB_G^+$ is the category having:
\begin{itemize}
\item objects: the finite $G$-sets
\item morphisms: equivalence classes of spans \[ \myspan{b}uc \] 
with composition given by pullbacks.
\end{itemize}
\end{defn}

Two spans $\inlinespan{b}uc$ and $\inlinespan{b}vc$ are equivalent if there
is a commutative diagram as follows:
\[ \xymatrix@R=0.5pc@C=1pc{
  & \gset{u} \ar[dr] \ar[dl] \ar[dd]^{\cong} & \\
  \gset{b} & & \gset{c} \\
  & \gset{v} \ar[ur] \ar[ul] & \\
}\]
Each hom set $\sB_G^+(\gset{b},\gset{c})$ has the structure of an abelian
monoid.  We may take the sum of $\inlinespan{b}uc$ and $\inlinespan{b}vc$
to be the span
\[ \myspan{b}{u\amalg v}c \]
We may thus apply the Grothendieck construction to the hom sets of
$\sB_G^+$.

\begin{defn}\label{def::burnside}
The \defword{Burnside category} $\sB_G$ is the category enriched over
$\ab$ having:
\begin{itemize}
\item objects: the objects of $\sB_G^+$
\item morphisms: $\sB_G(\gset{b},\gset{c})$ is the Grothendieck group of
$\sB_G^+(\gset{b},\gset{c})$.
\end{itemize}
\end{defn}

Since Mackey functors are additive functors, a Mackey functor $M$ over a
commutative ring $R$ is determined by its values on the orbits $G/H$.

\begin{defn}
The \defword{Burnside ring Mackey functor} $A_G$, abbreviated $A$ when the
group $G$ is implicit, is the represented functor $R\otimes\sB_G(-,G/G)$.
\end{defn}

The Burnside ring Mackey functor is so named because, when $R=\bZ$, its
value $A_G(G/H)$ at the orbit $G/H$ is the underlying Abelian group of the
classical Burnside ring $A(H)$.  In fact, the connection extends to the
ring structure as well.  The Day tensor product gives a monoidal structure
$\boxp$ on $\sB_G$, with unit $A$.  Explicitly, $\boxp$ is a left Kan
extension.  Given Mackey functors $M,N\colon \sB_G\to \rmod$, we can form
the external product 
\[ M\oboxp N \colon \sB_G\times\sB_G\to \rmod \colon
(\gset{b},\gset{c}) \longmapsto M(\gset{b})\otimes
N(\gset{c}), \]
and $M\boxp N$ is the left Kan extension of $M\oboxp  N$ along the
Cartesian product functor
$\times\colon \sB_G\times\sB_G \to \sB_G$:
\[ \xymatrix@C=3pc{
  \sB_G\times\sB_G \ar^{M\oboxp N}[r] \ar_{\times}[d] & \rmod \\
  \sB_G \ar@{.>}_{M\boxp N}[ur] \\
}\]
In other words, natural transformations from $M\boxp N$ to another Mackey
functor $P$ are the same as natural transformations from $M\oboxp N$ to
$P\circ\times$.
The upshot of this is that we have a notion of monoids in $\sB_G$, namely
Mackey functors $T$ together with a ``product map'' $T\boxp T\to T$
and a unit map $A\to T$ satisfying the usual diagrams.  These monoids are
known as \defword{Mackey functor rings} or \defword{Green functors}.  For
any Green functor $T$, the fact that $\boxp$ is a left Kan extension
implies that $T\boxp T\to T$ gives each $T(G/H)$ the structure of an
$R$-algebra.  The Burnside ring Mackey functor $A$ is, of course, a Green
functor, and the ring structure on each $A(G/H)$ agrees with the ring
structure on the classical Burnside ring $A(H)$.

The (unstable) orbit category embeds contravariantly and covariantly in
$\sB_G$.  Both embeddings are the identity on objects; a map of $G$-sets
$G/H\xrightarrow{\alpha} G/K$ is sent to either 
\[ \contraspan{G/H}{G/K}{\alpha} \hspace{10pt}\text{ or }\hspace{10pt} \cospan{G/H}{G/K}{\alpha} \]
as appropriate.  Furthermore, every morphism in $\sB_G$ can be written as
a composite of such morphisms:
\[ \xymatrix@R=0.5pc@C=1pc{
  & \gset{u} \ar[dl] \ar[dr] & \\
  \gset{b} & & \gset{c}\\}
 \hspace{10pt}=\hspace{10pt}
 \xymatrix@R=0.5pc@C=1pc{
  & \gset{u} \ar[dl] \ar@{=}[dr] & & \gset{u} \ar@{=}[dl] \ar[dr] & \\
  \gset{b} & & \gset{u} & & \gset{c}\\
}\]
Hence a Mackey functor $M$ determines and is determined by a pair of
contravariant and covariant coefficient systems which agree on objects (and
which satisfy certain compatibility diagrams encoding the composition in
$\sB_G$). As already mentioned, equivariant homology and cohomology
theories whose gradings can be extended to $RO(G)$ have coefficient systems
which extend to Mackey functors.

In particular, we can consider $RO(G)$-graded Bredon cohomology with
coefficients in the Burnside ring Mackey functor $A$; since $A$ is the unit
for $\boxp$, this is the natural choice of coefficients to consider.   In
what follows, coefficients in $A$ will be implicit and
omitted from the notation.  The $\rmod$ valued theory $\h^*$ can be
extended to a Mackey functor valued theory $\ulh^*$ as follows.  On
orbits, $\ulh^*(X)(G/K) := \h^*(G\times_K X) \cong \h^*((G/K) \times X)
\cong H_K^*(X)$.  On morphisms of type
\[\cospan{\gset{b}}{\gset{c}}{}\]
the required map $\ulh^*(X)(\gset{c})\to\ulh^*(X)(\gset{b})$ is induced on
cohomology by the map of spaces $\gset{b}\times X \to \gset{c}\times X$.
On morphisms of type 
\[\contraspan{\gset{b}}{\gset{c}}{}\]
the map $\ulh^*(X)(\gset{b})\to\ulh^*(X)(\gset{c})$ is induced by an
appropriate transfer map $\Sigma^V\gset{c_+} \to \Sigma^V\gset{b_+}$ and
the suspension isomorphism on cohomology.  Hence $\ulh^*$ ties together
information about the equivariant theories $H_K^*$ for all subgroups $K<G$.

\section{Equivariant classifying spaces}\label{sec::eqclassifying}

For the remainder of this paper, with the exception of classical structure
groups, groups named with Greek letters will be viewed as structure groups
and those named with Latin letters will be viewed as ambient groups of
equivariance.  Suppose that we are given a structure group $\Gamma$ and a
group of equivariance $G$.  Then, as discussed in \cite{AK} and many other
places, there is a notion of a principal $(G,\Gamma)$-bundle, namely a
projection to $\Gamma$-orbits $E\to B=E/\Gamma$ of a $\Gamma$-free
$(G\times\Gamma)$-space $E$.  Such equivariant bundles are classified by
universal principal bundles $E_G\Gamma\to B_G\Gamma$, where $E_G\Gamma$ is a space
whose fixed point sets $(E_G\Gamma)^\Lambda$ are empty when $\Lambda \subset
G\times\Gamma$ intersects $\Gamma = \triv\times \Gamma$ nontrivially and
contractible when $\Lambda \cap \Gamma = \triv$. 

As should be expected, for a fixed group $G$, the equivariant classifying
space construction can be made functorial; that is, there are functors
\begin{align*}
E_G &\colon \gset{Grp} \to \gtop\\
B_G &\colon \gset{Grp} \to \gtop.
\end{align*}
It will be helpful to pick particular functors $E_G$ and $B_G$, using the
categorical two-sided bar construction.  Given any groups $G$ and $\Gamma$, we
may take 
\begin{align*}
E_G \Gamma &:= B(T_\Gamma,\sO_{G\times\Gamma},O_{G\times\Gamma}) \\
B_G \Gamma &:= (E_G\Gamma)/\Gamma
\end{align*}
where $\sO_{G\times\Gamma}$ is the orbit category,
$O_{G\times\Gamma}\colon \sO_{G\times\Gamma}\to\spaces$ is given by viewing an
orbit $(G\times\Gamma)/\Lambda$ as a topological space, and $T_\Gamma\colon
\sO_{G\times\Gamma}\op \to\spaces$ is the functor which takes 
\[(G\times\Gamma)/\Lambda \longmapsto 
\begin{cases}
\bullet & \text{if }\Lambda\cap\Gamma=\triv\\
\emptyset & \text{otherwise}
\end{cases}\]
Since the functor $O_{G\times\Gamma}$ lands in $(G\times\Gamma)\textbf{-Top}$,
$E_G\Gamma$ is a $(G\times\Gamma)$-space, and it is easy to check that it has the
correct fixed points.

The bar construction $B(-,-,-)$ is a functor from the category of triples
$(T,\sC,S)$ to $\spaces$.  Here $\sC$ is a category and $S,T$ are
respectively a covariant and a contravariant functor $\sC\to\spaces$.  A
morphism $(T_1,\sC_1,S_1)\to (T_2,\sC_2,S_2)$ in the category of triples
consists of a functor $F\colon \sC_1\to\sC_2$ together with natural
transformations $S_1 \to S_2\circ F$ and $T_1\to T_2\circ F\op$.  It
follows that, for fixed $G$, we can make $E_G(-)$ into a functor
$\textbf{Grp}\to\spaces$ as follows.  Given a homomorphism $\varphi\colon
\Gamma_1\to \Gamma_2$, we apply $B(-,-,-)$ to the morphism of triples given by
the functor
\[ F \colon \sO_{G\times\Gamma_1} \to \sO_{G\times\Gamma_2} \colon (G\times \Gamma_1)/\Lambda
\longmapsto (G\times \Gamma_2)/{\left( (\id\times\varphi)(\Lambda) \right)},\]
with the obvious natural transformations $O_{G\times\Gamma_1} \to
O_{G\times\Gamma_2}\circ F$ and $T_{\Gamma_1}\to T_{\Gamma_2}\circ F\op$ (for the
latter, note that if $\Lambda\cap\Gamma_1 = \triv$, then also
$(\id\times\varphi)(\Lambda)\cap\Gamma_2 = \triv$).  Since the morphism
$E_G\Gamma_1 \to E_G\Gamma_2$ induced by $\Gamma_1\to\Gamma_2$ is $\Gamma_1$-equivariant,
there is an induced map $B_G\Gamma_1\to B_G\Gamma_2$, making $B_G(-)$ a functor.

The following result will be useful later.

\begin{prop}\label{prop::pullback}
Fix a group $G$.  Corresponding to any short exact sequence of structure
groups
\[ 1\to \Gamma \xrightarrow{\varphi} \Upsilon\xrightarrow{\psi} \Sigma \to 1 \]
there is a pullback square in the category of $G$-spaces
\[ \xymatrix@C=3pc{
  B_G \Gamma \ar[r] \ar_{/\Sigma}[d] \pullbackcorner
  & E_G \Sigma \ar^{/\Sigma\phantom{=/\Upsilon}}[d] \\
  B_G \Upsilon \ar[r]^{B_G\psi} & B_G \Sigma \\
}\]
\end{prop}

\begin{proof}
By functoriality of $E_G(-)$ and the definition of $B_G(-)$, the map $\psi$
induces a commutative diagram
\[ \xymatrix@C=3pc{
  E_G \Upsilon \ar[r]^{E_G\psi} \ar[d]_{/\Upsilon} 
  & E_G \Sigma \ar[d]^{/\Upsilon = /\Sigma} \\
  B_G \Upsilon \ar[r]^{B_G\psi} & B_G\Sigma \\
}\]
Further, the projection $E_G \Upsilon\to B_G \Upsilon = (E_G \Upsilon)/\Upsilon$
factors as 
\[ E_G \Upsilon \to (E_G \Upsilon)/\Gamma \xrightarrow{/\Sigma} (E_G \Upsilon)/\Upsilon \]
and since the action of $\Gamma$ on $E_G \Sigma$ is trivial, $E_G \Upsilon\to
E_G \Sigma$ also factors through $(E_G \Upsilon)/\Gamma \cong  B_G \Gamma$.  We
thus get the commutative square described in the proposition.  Viewing
$E_G\Sigma$ as a $G$-space via the inclusion $G\hookrightarrow
G\times\Sigma$, this is a diagram in the category of $G$-spaces.  The
square induces a homeomorphism on the fibers of the vertical maps, and
hence is a pullback in the category of $G$-spaces, as desired.
\end{proof}

\section{The cohomology of complex projective spaces}\label{sec::cpv}

\renewcommand{\G}{{C_{\!p}}}

Now specialize to the case $G = \G$, the cyclic group of order $p$, for
some prime $p>2$.  In \cite{lgl}, Lewis shows that certain $\G$-spaces,
including the complex projective spaces $\cpv{V}$ on an arbitrary
$\G$-representation $V$, have cohomology which is additively free as a
module over the cohomology of a point $\ulh^* :=
\ulh^*(\bullet)$.\footnote{\cite{lgl} also covers the case $p=2$; the
details are actually easier but slightly different.} The cohomology of a
point is a Green functor, with monoid structure induced by the diagonal map
$\bullet\to\bullet\times\bullet$, so it makes sense to talk about modules
over it.  As usual, such a module is called free if it is of the form
$\ulh^* \boxp \sB_\G(-,\gset{b})$, i.e.~a box product of $\ulh^*$ and a
representable Mackey functor.  Further, for the complex projective spaces
$\cpv{V}$, Lewis describes the product structure \[\ulh^*(\cpv{V}) \boxp
\ulh^*(\cpv{V}) \to \ulh^*(\cpv{V})\] which we will now summarize in one
case of interest.\footnote{This paper is part of the author's Ph.D. thesis,
which discusses Lewis's calculation in more detail and correct several
errors.}

Consider the complete universe $\cxuniverse$, a direct sum of countably
infinitely many copies of each irreducible complex representation of $\G$.
$\cpv{\cxuniverse}$ is a space of interest because, as will be discussed in
\autoref{sec::computations}, it is a model for the equivariant
classifying space $B_\G SO(2)$ and hence is related to equivariant
characteristic classes.  We can describe $\cxuniverse$ explicitly:
the irreducibles all have complex dimension one, and a generator $g$ of
$\G$ acts by rotation through $2\pi j/p$ radians in the complex plane, for
some integer $j$.  Hence there are $p$ irreducible representations
$\phi_0,\phi_1,\ldots,\phi_{p-1}$, with $\phi_0$ trivial and $\phi_1$
through $\phi_{p-1}$ nontrivial.  For convenience we may assume that
$\phi_j = \phi_1^j$ (i.e.~$\phi_1^{\otimes j}$) for each $j$.  Since $\phi_1^p =
\phi_0$, it then also makes sense to write $\phi_j = \phi_1^j$ for $j>p$.
We have $\phi_j = \phi_{j+p} = \phi_{j+2p} = \cdots$, and we can write
\[ \sU = \phi_0 + \phi_1 + \cdots + \phi_{p-1} + \phi_p + \cdots.\]

The dimensions of the generators of $\ulh^*(\cpv{\cxuniverse})$ are related to the
underlying real representations of the $\phi_j$; it is worth giving these
dimensions their own names.

\begin{defn}\label{def::omegas}
Let $\omega_0$ be the zero-dimensional representation $0$.  For each
positive integer $j$, define $\omega_j$ to be the underlying real
representation of
\[ \phi_j^{-1} \left( \phi_0 + \phi_1 + \cdots + \phi_{j-1} \right) \]
\end{defn}

Observe that $\omega_p$ is the underlying real representation of the
regular complex representation $\lambda$.  More generally, if $j = rp+j_0$
and $0 \le j_0 < p$, then $\omega_j$ is the underlying real representation
of $r\lambda \oplus \phi_{p-j_0} \oplus \cdots \oplus \phi_{p-1}$.

\begin{thm}[Lewis]\label{thm::addcpv}
As an $RO(G)$-graded module over $\ulh^*$,
\[ \ulh^*(\cpv{\cxuniverse}) \cong \bigoplus_{j\ge 0} \Sigma^{\omega_j}\!\!\left(
\ulh^* \boxp \sB_\G(-,\G/\G) \right) \cong \bigoplus_{j\ge 0}
\Sigma^{\omega_j} \ulh^* \]
where the $\omega_j$ are as given in \autoref{def::omegas}.
\end{thm}

Label the generators in dimensions $\omega_1$ through $\omega_{p-1}$ as
$D_1$ through $D_{p-1}$, and let $C$ be the generator in dimension
$\omega_p$.  It is shown in \cite{lgl} that these elements generate
$\ulh^*(\cpv{\cxuniverse})$ as an algebra over $\ulh^*$.

Explicitly, $C$ is a map of $\ulh^*$-modules $\Sigma^{\omega_p} (\ulh^*
\boxp \sB_\G(-,\G/\G)) \to \ulh^*(\cpv{\cxuniverse})$, which by Yoneda we may
identify with an element of
\[ \ulh^{\omega_p}(\cpv{\cxuniverse})(\G/\G) \cong \h^{\omega_p}(\cpv{\cxuniverse}), \]
and similarly for the $D_j$.
To describe the product structure on the graded Mackey functor
$\ulh^*(\cpv{\cxuniverse})$, it suffices to describe the levelwise products
\[ \ulh^*(\cpv{\cxuniverse})(\G/H) \otimes \ulh^*(\cpv{\cxuniverse})(\G/H) 
\to \ulh^*(\cpv{\cxuniverse})(\G/H) \]
for each $H<\G$.  However, $\G$ has only two subgroups.  If $H = \triv$,
then $\ulh^*(\cpv{\cxuniverse})(\G/e) \cong H^{|*|}(\bC P^\infty)$ with the expected
product structure.   If $H=\G$, we get $RO(\G)$-graded Bredon
cohomology, where $C$ and the $D_j$ live.

It remains to describe the products of the $D_j$ and $C$.  For notational
convenience, let $D_0 = 1$.

\begin{thm}[Lewis]\label{thm::multcpv}
As a commutative algebra over $\ulh^*$, $\ulh^*(\cpv{\cxuniverse})$ is generated by
elements $C$ in dimension $\omega_p$ and $D_j$ in dimension $\omega_j$ for
each $1\le j\le p-1$.  $C$ generates a polynomial subalgebra of
$\h^*(\cpv{\cxuniverse})$, and a complete set of additive generators of
$\ulh^*(\cpv{\cxuniverse})$ is given by the elements $\{ D_j C^n \}_{0\le
j\le p-1,\ n\ge 0}$.  
If $j\le k \le p-1$, then each product $D_j D_k = D_k D_j$ is given by a linear
combination over $\ulh^*$ of the elements $D_j$ through $D_{j+k}$, if
$j+k\le p-1$, and of $D_j$ through $D_{p-1}$ and $D_0 C = C$ through
$D_{j+k-p} C$ if $j+k\ge p$.
\end{thm}

The list of relations takes several pages to write out, and so will not be
reproduced here; see \cite{lgl} for details.
Some things to note are that $D_j \ne D_1^j$ (the dimensions are not
correct), and that $D_1^p$ is a nontrivial linear combination (over
$\ulh^*$) of the $D_j$ and $C$.

The elements $C$ and the $D_j$ are generators for
$\ulh^*{\cpv{\cxuniverse}}$ as a graded Green functor.  However, since
these generators all live at the $\G/\G$ level, there are no complications;
they give a set of algebra generators for the $RO(G)$-graded,
$\rmod$ valued theory $\h^*{\cpv{\cxuniverse}}$ as well.

One final note: Lewis uses $R=\bZ$ and the Burnside ring Mackey functor $A$
for all of his calculations.  However, if we instead used $R=\fq$ and
$A \cong A\otimes \fq$ 
(with the tensor product taken levelwise) for any prime $q\ne p$, all
arguments go through verbatim, and the algebra structure looks the same.
In the case $q = p$, however, the crucial \cite[Corollary 2.7]{lgl} fails,
and so the answer looks very different.

\section{Example: the cohomology of $B_\G O(2)$ }\label{sec::computations}

As before, let $\G$ be the cyclic group of order $p$.  In this
section, we will calculate the $RO(\G)$-graded Bredon cohomology
\[\h^*(B_\G O(2);A\otimes\fq)\] of the equivariant classifying space $B_\G
O(2)$ for odd primes $p\ne q$.  By way of motivation, we know from
\cite{lgl} what $\h^* B_\G S^1$ is, and historically $O(2)$ is often the
first test case to try after $S^1$.  Let $D_1$ through $D_p$ and $C$ be the
elements described in \autoref{thm::multcpv}.

\begin{thm}\label{thm::mythm}
Let $p,q$ be distinct odd primes.  Then, as an algebra over $\h^*$,
$\h^*(B_\G O(2);A\otimes\fq)$ is isomorphic to the subalgebra of
$\h^*(\cpv{\cxuniverse};A\otimes\fq)$ generated by the elements
$D_2,D_4,\ldots,D_{p-1},D_1 C,\ldots,D_{p-2}C,C^2$.
\end{thm}

By analogy with \autoref{ex::noneq}, we will approach \autoref{thm::mythm} via
the short exact sequence of structure groups
\[ 1\to SO(2) \to O(2) \to \cyctwo \to 1 \]
and the induced fibration $f\colon B_\G O(2) \to B_\G \cyctwo$, where again
$\cyctwo$ is the cyclic group of order 2.  We will first identify the
$\G$-action on the fibers of $f$, then explicitly describe the coefficient
systems $h^{V+t}(f,A)$ and $\coh_*(\widetilde{B_\G\cyctwo})$, and finally
prove the theorem.

\subsection{Identifying the fibers of $f\colon B_\G O(2)\to B_\G\cyctwo$}

We will begin by identifying models for the equivariant classifying spaces
under consideration.  Recall that nonequivariantly, the universal bundle
$E\cyctwo \to B\cyctwo$ has as a model $S^\infty \to \bR P^\infty$, where
$S^\infty = S(\bR^\infty)$ is the unit sphere in $\bR^\infty$ and $\bR
P^\infty = \bR P(\bR^\infty)$ is the infinite-dimensional real projective
space.   
In general, for any group $G$, let $\runiverse$ be a direct sum containing
countably infinitely many copies of each real representation of $G$.  Then
$S(\runiverse)\to \bR P(\runiverse)$ is a model for $E_G\cyctwo\to
B_G\cyctwo$.  The $G\times\cyctwo$ action on $S(\runiverse)$ comes from the
$G$ action on $\runiverse$ and the $\cyctwo$ action by multiplication by
$-1$.  When $G = \G$ is cyclic of prime order, however, we can choose a
simpler model. 

\begin{lem}\label{lem::rpinfty}
If $p$ is an odd prime, then $E_\G\cyctwo \to B_\G\cyctwo$ has as a model
$S^\infty\to\bR P^\infty$ with the trivial $\G$-action on both spaces.
\end{lem}

\begin{proof}
To verify this claim, it suffices to check the fixed-point sets of
$S^\infty$.  Since $\cyctwo$ acts freely on $S^\infty$, the fixed
points $(S^\infty)^\Lambda$ are certainly empty when $\Lambda\cap\cyctwo$
is nontrivial.  So we need only check that the fixed point set is
contractible whenever $\Lambda\cap\cyctwo$ is trivial.

Note that the subgroups $\Lambda\subset \G\times \cyctwo$ which intersect
$\cyctwo$ trivially are the ``twisted diagonal subgroups''
$\Lambda = \Delta_{\rho,H} = \{ (h,\rho(h)) | h\in H \}$, for $H$ a
subgroup of $\G$ and $\rho\colon H\to\cyctwo$ a homomorphism.  However,
since $p$ is an odd prime, the only homomorphism $H\to\cyctwo$ is the
trivial homomorphism, and so $\Delta_{\rho,H} = H\times\triv$.  This acts
trivially on $S^\infty$, so $(S^\infty)^\Lambda = S^\infty \simeq \bullet$,
as desired.  
\end{proof}

Similarly, $SO(2)$ is the circle $\bT$.  Letting $\cxuniverse$ again be the
direct sum of countably infinitely many copies of each irreducible complex
representation of $\G$, an analysis of the fixed-point sets of
$S(\cxuniverse)$ gives the following well-known result.

\begin{lem}\label{lem::cpinfty}
For any prime $p$, $E_\G SO(2) \to B_\G SO(2)$ has as a model
$S(\cxuniverse) \to \cpv{\cxuniverse}$.  The $\G\times SO(2)$ action on
$S(\cxuniverse)$ comes from the $\G$ action on $\cxuniverse$ and the usual
circle action on the complex plane. \hfill $\Box$
\end{lem}

We are now in a position to use \autoref{prop::pullback} for the short exact
sequence 
\[ 1\to SO(2) \to O(2) \xrightarrow{\text{det}} \cyctwo \to 1. \]
We have a pullback square in the category of $\G$-spaces
\[ \xymatrix@C=3pc{
  B_\G SO(2) \ar[r] \ar[d] \pullbackcorner & E_\G \cyctwo \ar[d] \\
  B_\G O(2) \ar[r] & B_\G \cyctwo \\
}\]
By \autoref{lem::rpinfty}, the $\G$-actions on $E_\G\cyctwo$ and $B_\G
\cyctwo$ are trivial.  It follows that $E_\G\cyctwo$ is $\G$-contractible,
and so we have proved the following about our map $f\colon B_\G O(2)\to
B_\G\cyctwo$.

\begin{lem}\label{lem::fibers}
For each point $x\in B_\G\cyctwo$, the fiber $f^{-1}(x)$ is $\G$-homotopy
equivalent to $B_\G SO(2) = \cpv{\cxuniverse}$.\hfill $\Box$
\end{lem}

\subsection{The local coefficient system $\sh^{V+t}(f,A)$}
Recall from \autoref{thm::bill} that the equivariant Serre spectral sequence
for a $\G$-fibration $f\colon E\to X$ has
\[ E_2^{s,t}(M,V) = \h^s(X;\sh^{V+t}(f,M)) \Longrightarrow
\h^{V+s+t}(E;M). \]
Choose the coefficient ring $R = \fq$, the finite field with $q$ elements,
for an odd prime $q\ne p$.  As in \autoref{thm::mythm}, let $M = A$; note
that $A$ here is a represented functor to $\fq$-mod, but we choose not to
explicitly make the $q$ part of the notation.  We must first analyze the
local coefficient systems
\[ \sh^{V+t}(f,A)\colon (H,x)\longmapsto \h^{V+t}(\G\times_H f^{-1}(x)). \]
We may start by taking a skeleton $\Pi$ of the category $\fund{B_\G\cyctwo}$.  
Again using \autoref{lem::rpinfty}, we see that each fixed-point set
$(B_\G\cyctwo)^H$ is nonempty and connected with fundamental group
$\cyctwo$, so $\Pi$ has two objects $(\G,x_0)$ and $(\triv,x_0)$.
Recalling that a map $(H,x)\to (K,y)$ consists of a $\G$-map $\alpha\colon
\G/H\to \G/K$ and a homotopy class of paths $[\gamma]$ from $x$ to
$\alpha^* y$, we see that there are two endomorphisms of $(\G,x_0)$. One of
these is the identity, and the other, $\kappa$, squares to the identity.
Similarly, there are two morphisms $(\triv,x_0)\to (\G,x_0)$, and
composition with $\kappa$ exchanges them.  Finally, the endomorphisms of
$(\triv,x_0)$ are in bijection with $\G\times\cyctwo$; when precomposing
with the two morphisms $(\triv,x_0)\to (\G,x_0)$, only the $\cyctwo$ factor
has an effect.  We may visualize $\Pi$ as follows:
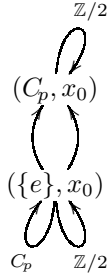
\[ \xymatrix@1{
(\G,x_0) \ar@(u,ur)[]^\cyctwo \\
(\triv,x_0) \ar@(d,dr)[]_{\cyctwo} \ar@(d,dl)[]^{\G} \ar@/^2ex/[u] \ar@/_2ex/[u] \\
}\]

We can then explicitly describe the coefficient system $\sh^{V+t}(f,A)$.
Let $\proj_{\G}$ be the projection $\G/\triv\to \G/\G$.

\begin{prop}\label{prop::h}
The functor $\sh^{V+t}(f,A)\colon \Pi\to\rmod$ takes
\begin{align*}
(\G,x_0) & \longmapsto \h^{V+t}(\cpv{\cxuniverse}; A) =
\ulh^{V+t}(\cpv{\cxuniverse})(\G/\triv)\\
(\triv,x_0) & \longmapsto H^{|V|+t}(\cpv{\cxuniverse};A(\G/e)) =
\ulh^{V+t}(\cpv{\cxuniverse})(\G/\G)
\end{align*}
The functor is determined on morphisms by the following: 
\begin{enumerate}
\item \label{item::down} the image of $(\proj_{\G},[c_{(\triv,x_0)}])$ is the image
of $\proj_{\G}$ in the underlying contravariant coefficient system of
the Mackey functor $\ulh^{V+t}(\cpv{\cxuniverse})$;

\item \label{item::endobottom} for each map $g\colon \G/\triv \to
\G/\triv$, the image of $(g,[c_{(\triv,x_0)}])$ is the image of $g$ in the
underlying contravariant coefficient system of
$\ulh^{V+t}(\cpv{\cxuniverse})$;

\item \label{item::endotop} Let $D_1$ through $D_{p-1}$ and $C$ be the
algebra generators described in \autoref{thm::multcpv}.  The nontrivial
automorphism of $(\G,x_0)$ acts by
the identity on the $D_{2k}$ and by multiplication by $-1$ on the
$D_{2k-1}$ and $C$.
\end{enumerate}
\end{prop}

This may be visualized by the diagram below.
\[ \xymatrix@1{
\h^{V+t}(\cpv{\cxuniverse}; A) \ar@/^2ex/[d] \ar@/_2ex/[d] \ar@(u,ur)[]^{\cyctwo} \\
H^{|V|+t}(\cpv{\cxuniverse};A(\G/e)) \ar@(d,dl)[]^{\G} \ar@(d,dr)[]_{\cyctwo} \\
} \]
Note that the second downward arrow is given by composing the maps of
(\ref{item::down}) and (\ref{item::endotop}).

\begin{proof}
\autoref{lem::fibers} identifies the value of $\sh^{V+t}(f;A)$ on objects.

For item (\ref{item::down}), 
the downward arrow induced by the morphism
$(\proj_{\G},[c_{(\triv,x_0)}])$ is simply the map on
cohomology induced by the space-level map
\[\G\times f^{-1}(x) \to f^{-1}(x).\]
This is the same as the map $\ulh^{V+t}(\cpv{\cxuniverse})(\G/\G) \to
\ulh^{V+t}(\cpv{\cxuniverse})(\G/\triv)$ induced by the span
\[ \cospan{\G/\triv}{\G/\G}{} \]
A similar argument identifies the map in item (\ref{item::endobottom}).

It remains to identify the $\cyctwo$ action.  Recall that any $B_G\Pi$ is of the
nonequivariant homotopy type of the classifying space $B\Pi$.  In
particular, our fibration $B_\G O(2)\to B_\G \cyctwo$ corresponds to the
nonequivariant fiber sequence
\[ BSO(2)\to BO(2)\xrightarrow{B\text{det}} B\cyctwo \]
We know that $H^*(BSO(2);\fq) \cong \fq[x]$, a polynomial algebra on a
generator $x$ in degree $2$, and that $\pi_1 B\cyctwo$ acts by $-1$ on $x$.
This determines the $\cyctwo$ action at the $\G/\triv$ level in the Mackey
functor $\ulh^{V+t}$, and thus at the $\G/\triv$ level in $\sh^{V+t}(f;A)$.
The algebra generators of $\ulh^*(\cpv{\cxuniverse})$ at the $\G/\G$ level
are $D_1$ through $D_{p-1}$ and $C$, where $D_j$ restricts to $x^j$ at the
$\G/\triv$ level and $C$ restricts to $x^p$.  It follows that the action of
$\cyctwo$ at the $\G/\G$ level must be by $-1$ on the elements $D_{2k-1}$
and $C$, and by the identity on the $D_{2k}$.
\end{proof}

\subsection{The local coefficient system $\coh_*(\widetilde{B_\G\cyctwo})$}

Recall that $\coh_*(\tilde{X})$ is the coefficient system $\Pi X\op\to\rmod$
which takes $(H,x)\longmapsto H_*(\widetilde{X^H}(x))$.  In our case, since
$B_\G\cyctwo \cong \bR P^\infty$ with trivial $\G$-action, it follows that 
$\coh_*(\widetilde{B_\G\cyctwo})$ is the constant functor at $H_*(S^\infty)
\cong H_*(\bullet)$.  We will continue to take our coefficient ring
$R=\fq$ for $q\ne p$, so $H_*(\bullet) \cong \fq$ concentrated in dimension
0.

One reason that $R=\fq$ is such a convenient choice in this example is that 
the constant functor at $H_*(\bullet)$ is projective.

{\renewcommand{\ground}{\fq}
\begin{prop}
The constant functor $\underline{\ground}\colon \Pi\op \to\rmod$ is a direct
summand of a representable functor and hence a projective object in the
category $[\Pi\op,\rmod]$.
\end{prop}

\begin{proof}
Consider the represented functor $\ground \Pi(-,(\trivo,x_0))$.  We see by
inspection that $\ground \Pi((K,x_0),(\trivo,x_0)) \cong \ground[\cyctwo]$ for both
possible values of $K$.  For an appropriate choice of basis, we can display this as
\[ \xymatrix@R=2pc@C=2pc{
\ground \oplus \ground
\ar@/_2ex/[d]_{\mymatrix{1&0\\0&1}} \ar@/^2ex/[d]^{\mymatrix{0&1\\1&0}}
\ar@(u,ur)[]^{\mymatrix{0&1\\1&0}} \\
\ground \oplus \ground
\ar@(d,dr)[]_{\mymatrix{0&1\\1&0}} \ar@(d,dl)[]^{\text{triv}} \\
}\]
That is, the nontrivial element of $\cyctwo$ acts by interchanging the
basis elements, on both the top and the bottom.  The action of $\G$ on the bottom
is trivial, and the downward maps behave as shown.  If we take the new
basis given by the change-of-coordinates matrix $\mymatrix{1&\phantom{-}1\\1&-1}$
(using the fact that $q\ne 2$), we see that $\ground \Pi(-,(\trivo,x_0))$
breaks up as the direct sum of two functors, one of which is our constant
functor $\underline{\ground}$.
\[ \xymatrix@R=0.5pc@C=1pc{
  & \ground \ar@(u,ur)[]^{\id} \ar@/_2ex/[dd]_{\id} \ar@/^2ex/[dd]^{\id} & & 
    \ground \ar@(u,ur)[]^{-1} \ar@/_2ex/[dd]_{\id} \ar@/^2ex/[dd]^{-1} \\
\ground \Pi(-,(\trivo,x_0))\hspace{10pt}\cong\hspace{10pt}  & & \hspace{10pt}\oplus\hspace{10pt} & \\
  & \ground \ar@(d,dr)[]_{\id} \ar@(d,dl)[]^{\text{triv}} & &
    \ground \ar@(d,dr)[]_{-1} \ar@(d,dl)[]^{\text{triv}} \\
}\]
\end{proof}
}



\subsection{The calculation of $\h^*(B_\G O(2);A\otimes\fq)$}

We are now prepared to prove \autoref{thm::mythm}.  Fix odd primes $p\ne q$.
We will continue to make heavy use of the identification of $B_\G \cyctwo$
in \autoref{lem::rpinfty}.

\begin{proof}[Proof of \autoref{thm::mythm}]
We will use the equivariant Eilenberg spectral sequence to identify the
$E_2$ page of the Serre spectral sequence for $f\colon B_\G O(2)\to
B_\G\cyctwo$ and then show that the Serre spectral sequence collapses with
no extension problems.

Since $\widetilde{B_\G\cyctwo}$ is the constant functor at $S^\infty$, the relevant
equivariant Eilenberg spectral sequence in this case is
\[ \Ext_\Pi^{u,v} \left( \coh_*(S^\infty),\sh^{V+t}(f;A) \right) \Longrightarrow
\h^{u+v}(B_\G\cyctwo;\sh^{V+t}(f;A)). \]
As before, in the $E_2$ term, $u$ is the homological degree and $v$ is the
internal grading on $\coh_*$.  Since $\coh_v(S^\infty)$ is either $0$ or
$\underline{\fq}$, both of which are projective,
$\Hom_\Pi(\coh_v(S^\infty),-)$ is exact, and so all $\Ext$ terms with $u>0$
vanish.  It follows that the spectral sequence collapses at $E_2$ with no
extension problems, and so the $E_2$ terms of the Serre spectral sequence
are given by
\[ \h^s(B_\G\cyctwo;\sh^{V+t}(f;A)) \cong
\Hom_\Pi \left( \coh_s(S^\infty),\sh^{V+t}(f;A) \right). \]
The homology of $S^\infty$ vanishes for $s>0$, so in fact the Serre
spectral sequence also collapses with no extension problems.

$B_\G\cyctwo$ has a trivial $\G$ action and is $\G$-connected, meaning that
\autoref{prop::eqhzero} and \autoref{prop::gconnected} apply.  Thus
we may identify
\[ \h^{V+t}(B_\G O(2);A) \cong \h^0(B_\G\cyctwo;\sh^{V+t}(f;A))
\hookrightarrow \h^{V+t}(B_\G SO(2);A)\]
as algebras over the cohomology of a point.

More specifically, we have
\[ \h^{V+t}(B_\G O(2);A) \cong \Hom_\Pi(\underline{\fq},\sh^{V+t}(f;A)). \] 
Our coefficient systems take values in $(\fq)$-mod, so for any $\sN$ and
any element ${\eta\nobreak\in\Hom_\Pi(\underline{\fq},\sN)}$, $\eta$ factors through
the ``fixed subfunctor of $\sN$,''  i.e.~the subfunctor
\[ \xymatrix@1{
\sN(\G,x_0)^{\cyctwo} \ar[d] \\
\sN(\triv,x_0)^{\G\times\cyctwo}.
}\]
Note the two downward arrows must give the same map, and so a single arrow has
been drawn above.  As already observed, our $\Pi$ has a weakly terminal
object, and so a map in $\Hom_\Pi(\underline{\fq},\sN)$ is in fact
determined by choosing an element of $\sN(\G,x_0)^{\cyctwo}$.  By
inspection of the structure of $\Pi = \fund{B_\G\cyctwo}$, we see that
every element of $\sN(\G,x_0)^{\cyctwo}$ defines a map in
$\Hom_\Pi(\underline{\fq},\sN)$, as well.  
In other words, we have demonstrated that 
\[ \h^{V+t}(B_\G O(2);A) \cong \h^{V+t}(B_\G SO(2);A)^{\cyctwo}\]
for each $V$ and $t$, and hence
\[ \h^{*}(B_\G O(2);A) \cong \h^{*}(B_\G SO(2);A)^{\cyctwo}.\]
By \autoref{thm::multcpv} and \autoref{prop::h}, it follows that
$\h^{*}(B_\G O(2);A)$ is the subalgebra of $\h^{*}(B_\G SO(2);A)$ generated
by the elements $D_{2k}$, $D_{2k-1} C$, and $C^2$, the generators
restricting to an even power of the nonequivariant generator $x$ of
$H^*(\mathbb{C}P^\infty)$.
\end{proof}

In fact, closer examination shows that the Green functor $\ulh^*(B_\G
O(2))$ is a sub-Green functor of $\ulh^*(\cpv{\sU}) = \ulh^*(B_\G SO(2))$,
again on the generators $D_{2k}$, $D_{2k-1} C$, and $C^2$.

\bibliographystyle{alpha}
\bibliography{localcoeff}

\end{document}